\documentclass[11pt,reqno]{amsart}

\usepackage[T1]{fontenc}

\usepackage{lmodern}
\usepackage{xcolor}
\usepackage{enumitem}
\usepackage{microtype}
\usepackage{amsmath,amssymb,amscd,mathtools,mathrsfs}

\usepackage[colorlinks=true,linkcolor=blue!45!black,citecolor=blue!45!black,
urlcolor=blue!45!black]{hyperref}
\usepackage{tikz-cd}

\setlist[enumerate]{leftmargin=2.2em,itemsep=.25em,topsep=.35em}
\setlist[itemize]{leftmargin=2em,itemsep=.2em,topsep=.35em}

\newtheorem{theorem}{Theorem}[section]
\newtheorem{proposition}[theorem]{Proposition}
\newtheorem{lemma}[theorem]{Lemma}
\newtheorem{corollary}[theorem]{Corollary}
\theoremstyle{definition}

\newtheorem{definition}[theorem]{Definition}
\theoremstyle{remark}

\newcommand{\La}{\Lambda}

\newcommand{\Z}{\mathbb Z}
\newcommand{\Gm}{\mathbb G_m}
\newcommand{\Fuk}{\mathcal F}
\newcommand{\A}{\mathcal A}
\newcommand{\C}{\mathcal C}
\newcommand{\Perf}{\operatorname{Perf}}

\newcommand{\DPic}{\operatorname{DPic}}
\newcommand{\Auteq}{\operatorname{Auteq}}

\newcommand{\Hom}{\operatorname{Hom}}

\newcommand{\HH}{\operatorname{HH}}
\newcommand{\im}{\operatorname{im}}
\newcommand{\id}{\operatorname{id}}
\newcommand{\ch}{\operatorname{ch}}
\newcommand{\ev}{\operatorname{ev}}

\newcommand{\rk}{\operatorname{rk}}
\newcommand{\Pic}{\operatorname{Pic}}

\newcommand{\Dih}{\operatorname{Dih}}
\newcommand{\vcd}{\operatorname{vcd}}
\newcommand{\PiShift}{\Pi}

\title[Derived Picard groups of closed surfaces]
{The derived Picard group of Fukaya categories of closed surfaces}

\author[D. Wu]{Dongjian Wu}
\address{D. Wu: Shanghai Institute for Mathematics and Interdisciplinary Sciences (SIMIS), Shanghai 200433, China \& Research Institute of Intelligent Complex Systems, Fudan University, Shanghai 200433, China}
\email{wdj@simis.cn}

\author[N. Zhang]{Nantao Zhang}
\address{N. Zhang: School of Mathematics, Sun Yat-sen University, Guangzhou 510275, China}
\email{zhangnt@mail.sysu.edu.cn}


\begin{document}

\begin{abstract}
Let $\Sigma_g$ be a closed oriented surface of genus $g\geq2$, and let
$\mathscr F_g$ be its split-closed, strictly unobstructed, two-periodic
Fukaya category over the complex Novikov field $\Lambda$.  We determine
the derived Picard group
\[
 \DPic_{\Lambda}(\mathscr F_g)
 \cong
 \bigl(H^1(\Sigma_g;\Lambda^\times)
 \rtimes\pi_0\operatorname{Diff}^+(\Sigma_g)\bigr)
 \times\Z/2\Z.
\]
This proves the enhanced form of the conjecture of Auroux and Smith \cite{AurouxSmith}. As a corollary, we obtain that the derived Picard group itself determines the genus of the surface.
\end{abstract}

\maketitle

\section{Introduction}
Fukaya categories provide algebraic invariants of symplectic manifolds,
and for surfaces they admit a particularly concrete description
\cite{SeidelBook}.  Objects are curves equipped with local systems,
and the higher products are defined by counting polygons.  The closed
higher-genus case was treated in \cite{Abouzaid}, while a
local-to-global description in terms of pairs of pants was later
obtained in \cite{PascaleffSibilla}.  For surfaces with boundary,
related topological models and their connection with gentle algebras
appear in \cite{HKK,LekiliPolishchuk}.  Throughout this paper, we use
the strictly unobstructed two-periodic category of \cite{AurouxSmith},
with the same Novikov field and the same convention for brane
structures.

Let $\Sigma_g$ be a closed connected oriented surface of genus
$g\geq2$, equipped with an area form of total area one.  We write
\[
 \mathscr F_g=\Perf\bigl(\Fuk(\Sigma_g;\La)\bigr)
\]
for the split-closed Fukaya category.  Its cohomological category is
\[
 D^\pi\Fuk(\Sigma_g;\La)=H^0(\mathscr F_g).
\]

Symplectomorphisms act naturally on $\mathscr F_g$, which induces a map from the mapping class group to categorical symmetries.  We then ask which enhanced autoequivalences are induced by mapping classes and
what the kernel of this map is. The enhanced symmetry group we consider in this paper is the \emph{derived Picard group}
\[
 G_g=\DPic_\La(\mathscr F_g),
\]
which consists of invertible perfect
$\mathscr F_g$--$\mathscr F_g$ bimodules, with product given by
convolution.  Equivalently, it is the group of Morita
autoequivalences together with their $A_\infty$ enhancement.  This is
the usual derived Morita point of view on enhanced categories
\cite{KellerDG,Toen}. 

Our starting point is the geometrization result of
\cite{AurouxSmith}: a spherical object with non-zero Chern character is
represented by a homologically essential simple closed curve carrying a
rank-one local system, and Floer cohomology recognizes when two
such curves have geometric intersection one.  The category therefore
recovers the graph of non-separating curves introduced in
\cite{Schmutz}.  Its rigidity gives a split mapping-class quotient
\[
 \Auteq\bigl(D^\pi\Fuk(\Sigma_g;\La)\bigr)
       \longrightarrow \Gamma_g,
 \qquad \Gamma_g=\pi_0\operatorname{Diff}^+(\Sigma_g).
\]
The kernel was conjectured in \cite{AurouxSmith} to be generated by
flux and by tensoring with flat rank-one local systems, which are recorded by the group
\[
 H_g=H^1(\Sigma_g;\La^\times)
     =\Hom\bigl(H_1(\Sigma_g;\Z),\La^\times\bigr).
\]
Thus the conjectural answer is
\begin{equation}\label{eq:AS-conjecture}
 \Auteq\bigl(D^\pi\Fuk(\Sigma_g;\La)\bigr)
 \overset{?}{\cong}H_g\rtimes\Gamma_g .
\end{equation}
The two-periodic category also has the central shift $[1]$.  It acts
trivially on the unoriented curve graph but changes the sign of the
Chern character.  It must therefore be separated from the character
subgroup.

Our main result determines the enhanced group and, in particular, the
kernel of the mapping-class quotient. 

\begin{theorem}\label{thm:main}
Let $g\geq2$.  After choosing the geometric section $s_g$, there is an
isomorphism
\begin{equation}\label{eq:main-group}
 G_g\cong
 \bigl(H_g\rtimes\Gamma_g\bigr)\times\langle\PiShift\rangle,
 \qquad \PiShift^2=1,
\end{equation}
where the first semidirect product is defined by
\[
 (f\cdot\rho)(\alpha)=\rho(f_*^{-1}\alpha).
\]
In particular, $H_g$ is normal in $G_g$.  The two graph quotients have
the following kernels.
\begin{enumerate}[label=\textup{(\roman*)}]
\item If $g\geq3$, the Floer--Schmutz quotient
      $\Psi_g\colon G_g\twoheadrightarrow\Gamma_g$ has kernel
      \[
       \ker\Psi_g=H_g\times\langle\PiShift\rangle.
      \]
      \smallskip
\item If $g=2$ and $\iota\in\Gamma_2$ is the hyperelliptic involution,
      the oriented lift $\Psi^{\mathrm{or}}_2$ satisfies
      $\Psi^{\mathrm{or}}_2(\PiShift)=\iota$ and
      \[
       \ker\Psi^{\mathrm{or}}_2
       \cong H_2\rtimes_{\rho\mapsto\rho^{-1}}\Z/2
       =\Dih(H_2).
      \]
      A reflection is represented by
      $\PiShift s_2(\iota)^{-1}$.
\end{enumerate}
\end{theorem}

For $g\geq3$, the parity shift lies in the kernel of the ordinary
Floer--Schmutz quotient.  Genus two is slightly different.  Reversing all
oriented vertices is the hyperelliptic involution, and the oriented
quotient sends $[1]$ to this class.  This is the reason for the
generalized dihedral kernel in part~\textup{(ii)}.  After quotienting by
the parity shift, equation \eqref{eq:main-group} gives precisely the
enhanced version of \eqref{eq:AS-conjecture}.

The formula
\[
 \dim_\La\HH^{\mathrm{odd}}(\mathscr F_g)=2g,
\]
shows that the enhanced Morita class of $\mathscr F_g$ determines the genus.
The abstract group $G_g$ has the same property. To include the torus,
let $G_1$ be the derived Picard group of the graded Fukaya category of
$T^2$, modulo the square of the shift.

\begin{corollary}
\label{cor:abstract-reconstruction}
Let $g,h\geq1$.  If $G_g\cong G_h$ as abstract groups, then $g=h$.
\end{corollary}

We outline the main steps of the proof. First, choose a filling
$A_{2g}$-chain $E_1,\ldots,E_{2g}$. Its homology classes form an
integral basis, and the corresponding objects split-generate
$\mathscr F_g$. A kernel element may be corrected by one character in
$H_g$ and by a common parity shift. After this correction, it fixes
every object in the chain. The one-dimensional Floer groups between
adjacent objects then force its first Taylor coefficient to be a single
scalar $d\in\La^\times$.

To determine $d$, we vary a curve brane in its algebraic family of
rank-one local systems. The parameter simultaneously records holonomy
and flux. The relative Floer complex of the image family with a target
rank-one family has support on the graph of an algebraic automorphism
of $\Gm$. Differentiating near the reference brane yields
$d=\pm1$. The negative sign is realized by the hyperelliptic involution
composed with the parity shift.

It remains to remove the higher Taylor coefficients.  The joint
evaluation map
\[
 \HH^1(\mathscr F_g)\longrightarrow
 \bigoplus_{i=1}^{2g}HF^1(E_i,E_i)
\]
is injective.  A class represented in arity at least two has zero
evaluation on every $E_i$, and hence the positive-arity part
$W_2\HH^1$ vanishes.  Hochschild integration now implies that an
object-fixing $A_\infty$-autofunctor with identity first Taylor
coefficient is weakly equivalent to the identity. 

The paper is organized as follows.  Section~\ref{sec:input} fixes the
Fukaya-category conventions and recalls the Schmutz graph, the
mapping-class quotient, and the algebraic families of rank-one branes.
Section~\ref{sec:rigidity} proves rigidity for an autoequivalence fixing
the generating chain.  Section~\ref{sec:proof} computes the two kernels
and the full group, and then proves
Corollary~\ref{cor:abstract-reconstruction}.  The appendix records only
the Hochschild integration result used in the rigidity argument.

\subsection*{AI disclosure}
The authors developed the overall strategy, while ChatGPT aided in some proof verification, language editing. The authors take full responsibility for all mathematical claims, arguments, and conclusions, as well as for any remaining errors.

\subsection*{Acknowledgements}
We would like to thank Yu-Wei Fan and Shizhuo Zhang for their continuous support and discussions. D. W. would particularly like to thank Wen Chang for discussions on related topics concerning derived Picard groups.

\section{The Fukaya category and the Schmutz graph}\label{sec:input}
In this section, we recall the surface Fukaya category and the geometric input
used in the proof.  We mainly follow the conventions of
\cite[Sections~2, 6 and 7]{AurouxSmith}.

\subsection{The Fukaya category of a surface}

We use the complex universal Novikov field
\[
 \La=\left\{
 \sum_{j=0}^{\infty}a_jq^{\lambda_j}\ \middle|\
 a_j\in\mathbb C,\ \lambda_j\in\mathbb R,\
 \lambda_j\longrightarrow+\infty
 \right\}.
\]
It is algebraically closed of characteristic zero and carries the
valuation
\[
 \operatorname{val}\colon\La\longrightarrow\mathbb R\cup\{+\infty\},
 \qquad
 \operatorname{val}\!\left(\sum_j a_jq^{\lambda_j}\right)
   =\min\{\lambda_j:a_j\ne0\}.
\]
Put $\La_{\geq0}=\operatorname{val}^{-1}[0,+\infty]$ and
$U_\La=\operatorname{val}^{-1}(0)$.  Thus every $z\in\La^\times$ has a
unique expression
\begin{equation}\label{eq:factorization}
 z=q^s u,
 \qquad s=\operatorname{val}(z)\in\mathbb R,
 \qquad u\in U_\La.
\end{equation}
We write $\Gm$ for the multiplicative group over $\La$, so that
$\Gm(\La)=\La^\times$.

An uncurved two-periodic $A_\infty$-category $\C$ has $\Z/2$-graded
morphism spaces and operations
\[
 \mu^d_\C\colon
 \C(X_{d-1},X_d)\otimes\cdots\otimes\C(X_0,X_1)
 \longrightarrow \C(X_0,X_d)[2-d],
 \qquad d\geq1,
\]
which satisfy the $A_\infty$-relations. Here uncurved means that
$\mu^0=0$. We work throughout with cohomologically unital categories.

Equip $\Sigma_g$ with an area form $\omega$ of total area one.  The
strictly unobstructed Fukaya category
$\Fuk(\Sigma_g;\La)$ is defined as follows.  Its objects are Lagrangian
branes
\begin{equation*}
 X=(\iota\colon S^1\looparrowright\Sigma_g,\xi,\mathfrak s).
\end{equation*}
The curve $\iota(S^1)$ is homotopically essential.  We assume that all its
self-intersections are transverse.  The local system $\xi$ is a finite-rank flat unitary
$\La$-local system on $S^1$.  Unitary means that its monodromy preserves
the valuation filtration.  In rank one the holonomy lies in $U_\La$.
The symbol $\mathfrak s$ denotes a spin structure, which fixes the signs in
the polygon counts.  We choose an orientation whenever the homology class
of a brane is used.  The formal shift $[1]$ changes its parity.

The obstruction term $\mu^0_X$ is the signed Novikov-weighted count of
holomorphic teardrops with boundary on $\iota(S^1)$.  Strict
unobstructedness means
\(
 \mu^0_X=0\). For a
homotopically non-trivial immersed curve this condition is equivalent to
the absence of teardrops.  It is also equivalent to the assertion that
every lift of the curve to the universal cover is embedded
\cite[Lemma~2.2]{AurouxSmith}.  In particular every homotopically
essential embedded curve is strictly unobstructed.

Suppose that two branes $X_0=(\gamma_0,\xi_0)$ and
$X_1=(\gamma_1,\xi_1)$ meet transversely.  Their morphism space is the
Floer cochain complex
\begin{equation}\label{eq:Floer-complex}
 CF^*(X_0,X_1)=
 \bigoplus_{p\in\gamma_0\cap\gamma_1}
 \Hom_\La\bigl((\xi_0)_p,(\xi_1)_p\bigr).
\end{equation}
The local intersection signs and the brane data determine its
$\Z/2$-grading.  The differential $\mu^1$ counts immersed bigons.  The
higher operation $\mu^d$ counts rigid immersed polygons with $d+1$
convex corners.  Each polygon is weighted by its sign, by the parallel
transport maps of the local systems, and by the Novikov factor determined
by its symplectic area.  Standard perturbation data give the same
definition when the curves are not transverse.  Since every object is
strictly unobstructed, these operations define an uncurved two-periodic
$A_\infty$-category. 

Let $\operatorname{Tw}(\C)$ be the category of finite twisted complexes
over an $A_\infty$-category $\C$.  Its idempotent completion is denoted by
\[
 \operatorname{Tw}^\pi(\C)=\Perf(\C).
\]
The split-closed Fukaya category is given by $\mathscr F_g\coloneq\Perf\bigl(\Fuk(\Sigma_g;\La)\bigr)$. A collection split-generates $\mathscr F_g$ when its twisted complexes and
their direct summands contain every object of $\mathscr F_g$.  Whenever a
finite full subcategory is used below, we replace it by a minimal strictly
unital model.  In such a model $\mu^1=0$, so the first Taylor coefficient
of an $A_\infty$-functor is its induced map on Floer cohomology.

Recall that the derived Picard group  $G_g=\DPic_\La(\mathscr F_g)$
 consists of isomorphism classes of invertible perfect
$\mathscr F_g$--$\mathscr F_g$ bimodules.  Its product is the derived
tensor product and convolution with a bimodule gives an enhanced Morita
autoequivalence. An abstract exact functor of the cohomological category
does not by itself provide a bimodule, Taylor coefficients, or algebraic
base change.  This is the reason for working with the enhanced group.

We now fix notation for this action.  Given $F\in G_g$, choose an invertible
perfect bimodule $M_F$ representing its class, and denote the Morita
autoequivalence obtained by convolution with $M_F$ by $\Phi_F$.  For an
object $X$ we write
\[
 F(X):=\Phi_F(X),
\]
which is well defined up to isomorphism.  A
different representative of $F$ gives a naturally isomorphic functor. We write $F_*$ for the induced maps on Floer and Hochschild cohomology.

At several points the restriction of $\Phi_F$ to a finite full
$A_\infty$-subcategory is represented by a strictly unital
$A_\infty$-functor, which we again denote by
\[
 F=(F^0,F^1,F^2,\ldots).
\]
Here $F^0$ is the map on objects.  For $d\geq1$, the term $F^d$ is the
$d$-linear Taylor coefficient on morphism spaces.  In particular,
\[
 F^1\colon\hom(X,Y)\longrightarrow\hom(F^0(X),F^0(Y))
\]
is a chain map.  On a minimal model, where $\mu^1=0$, it is the map $F_*$
on Floer cohomology. 

\subsection{Floer cohomology and spherical objects}

For isotopy classes of simple closed curves, let
\[
 i(\gamma_0,\gamma_1)\coloneq
 \min_{\gamma'_j\simeq\gamma_j}|\gamma'_0\cap\gamma'_1|
\]
be their geometric intersection number.  Curves may be moved to minimal
position without changing the corresponding Fukaya objects.  If their
algebraic intersection is non-zero, no cancellation occurs in the
rank-one Floer complex. By \cite[Corollary~2.11]{AurouxSmith}, for rank-one branes on curves $X_0,X_1$, we have
\begin{equation}\label{eq:intersection}
 \rk_\La HF^*(X_0,X_1)=i(\gamma_0,\gamma_1).
\end{equation}
An object $X\in\mathscr F_g$ is \emph{spherical} when
\[
 HF^*(X,X)\cong H^*(S^1;\La)
\]
with its non-degenerate degree-one pairing.  An embedded curve carrying a
rank-one local system is spherical.  Conversely, a spherical object with
non-zero Chern character is quasi-isomorphic to such a brane
\cite[Theorem~1.1]{AurouxSmith}.

 For a
perfect object $X$, its identity endomorphism has a categorical trace
\begin{equation*}
 \ch(X)\coloneq\operatorname{tr}_X(\id_X)\in\HH_0(\mathscr F_g).
\end{equation*}
The open--closed map sends Hochschild chains to homology classes in the
surface.  For an oriented embedded brane $(\gamma,\xi)$ one has
\[
 \operatorname{OC}\bigl(\ch(\gamma,\xi)\bigr)
 =\rk(\xi)[\gamma].
\]
Thus a simple closed curve has non-zero Chern character exactly when it is
non-separating.  Equivalently, its complement in $\Sigma_g$ is connected. By \cite[Corollary~2.16]{AurouxSmith} there exist open--closed and closed--open maps which give the identifications
\begin{equation}\label{eq:HH-identification}
 \operatorname{OC}\colon\HH_0(\mathscr F_g)\xrightarrow{\sim}
 H_1(\Sigma_g;\La),
 \qquad
 \operatorname{CO}\colon H^1(\Sigma_g;\La)\xrightarrow{\sim}
 \HH^1(\mathscr F_g)
\end{equation} 
If $E$ is an
oriented embedded curve and $a\in H^1(\Sigma_g;\La)$, then
\begin{equation}\label{eq:evaluation}
 \ev_E(a)=\langle a,[E]\rangle u_E,
 \qquad
 u_E\in HF^1(E,E).
\end{equation}
Here $u_E$ is the logarithmic deformation of the rank-one local system on
$E$.

\subsection{Flux, holonomy, and rank-one families}
Let $\gamma_t$, with $0\leq t\leq1$, be an isotopy of an oriented curve.
Its relative flux is the signed area swept by the isotopy.  Thus
\begin{equation*}
 \operatorname{Flux}(\gamma_t)
 =\int_{[0,1]\times S^1}\Gamma^*\omega,
 \qquad \Gamma(t,x)=\gamma_t(x).
\end{equation*}
An isotopy of flux zero, together with parallel transport of the local
system and spin structure, produces a quasi-isomorphic Fukaya object
\cite[Lemma~2.12]{AurouxSmith}. If $\gamma$ is non-separating, embedded representatives with
arbitrary real relative flux exist \cite[Lemma~2.10]{AurouxSmith}.

Fix an oriented isotopy class of a non-separating curve and fix a parity.
Choose a reference rank-one brane $X_\gamma$.  Any other rank-one brane in
this class is obtained from $X_\gamma$ by changing two pieces of data.  The
support may move through an isotopy of real flux $s$, and the local system
may acquire holonomy $u\in U_\La$.  Formula
\eqref{eq:factorization} combines them into
\( z=q^s u\in\La^\times\). Multiplication of $z$ defines a free and transitive action of $\Gm$ on
these branes.  We denote the resulting $\Gm$-torsor by
$\mathfrak B_\gamma$.  The choice of $X_\gamma$ identifies this torsor
with $\Gm$.  A different reference brane replaces the coordinate by
$z\mapsto cz$ for some $c\in\La^\times$.

If the orientation is reversed, the signed flux changes from $s$ to
$-s$ and the holonomy around the oriented circle changes from $u$ to
$u^{-1}$.  After the reference brane on the reversed curve is chosen to
be the dual of $X_\gamma$, the coordinate therefore changes by
\begin{equation*}
 z=q^su\longmapsto q^{-s}u^{-1}=z^{-1}.
\end{equation*}

Choose $a\in H^1(\Sigma_g;\Z)$ with
$\langle a,[\gamma]\rangle=1$, and set
$R\coloneq\La[z,z^{-1}]$.  The rational $\Gm$-action associated
to $a$ \cite[Proposition~6.6]{AurouxSmith} gives the
orbit of $X_\gamma$ as a perfect family of modules
over $R$ \cite[Section~7.2]{AurouxSmith}.  We denote it by
\begin{equation*}
 \mathcal U_\gamma\in
 \Perf(\mathscr F_g\otimes_\La R).
\end{equation*}
For $z_0=q^s u$, where $u\in U_\La$, its fibre is
represented by the image of $X_\gamma$ under an isotopy
of flux $s a$, equipped with the unitary twist whose
holonomy on $\gamma$ is $u$.  Thus $z_0$ records both
flux and holonomy.  In particular, it is not in general
the monodromy of a local system on the fixed support.

We write
\[
 X_\gamma(z_0)=
 \mathcal U_\gamma\otimes_R^{\mathbb L}R/(z-z_0),
 \qquad z_0\in\La^\times.
\]
The family is based at $X_\gamma(1)\simeq X_\gamma$ and its infinitesimal deformation is given by
\begin{equation*}
 \operatorname{KS}(\mathcal U_\gamma)
 =\operatorname{CO}(a)|_{X_\gamma}\otimes d\log z
 =u_\gamma\otimes d\log z,
\end{equation*}
which is independent of the particular picture used to
realize the flux isotopy.

\subsection{The Schmutz graph and the mapping-class quotient}

The graph below was introduced in \cite[Section~2]{Schmutz}.  We use the
terminology of \cite[Section~7.2]{AurouxSmith}.

\begin{definition}
The \emph{Schmutz graph} $\mathcal S(\Sigma_g)$ has one vertex for every
isotopy class of non-separating simple closed curves in $\Sigma_g$.
Distinct vertices $[\gamma]$ and $[\delta]$ are joined by an edge if and only if \(i(\gamma,\delta)=1\). A curve and the same curve with the opposite
orientation define the same vertex.
\end{definition}

By \cite{Schmutz}, the automorphism group of
$\mathcal S(\Sigma_g)$ is isomorphic to the extended mapping class group for
$g\geq3$.  In genus two the action has kernel generated by the
hyperelliptic involution $\iota$.  The orientation-preserving subgroup is
$\Gamma_g=\pi_0\operatorname{Diff}^+(\Sigma_g)$.

We now describe how this graph is recovered from $\mathscr F_g$.  Let
$\mathcal P_g$ be the collection of spherical objects with non-zero Chern
character.  For $X,X'\in\mathcal P_g$, set
\begin{equation}\label{eq:categorical-relation}
 \begin{split}
 X\sim X'\quad\text{ if }\quad
 &\rk HF^*(X,Y)=1\ \Longleftrightarrow\
   \rk HF^*(X',Y)=1\\[-2pt]
 &\hspace{38mm}\text{for every }Y\in\mathcal P_g.
 \end{split}
\end{equation}

The vertices of the \emph{Floer--Schmutz graph} $\Upsilon(\mathscr F_g)$ are given by the
equivalence classes defined above.  Two distinct classes are joined when
the Floer cohomology between representatives has rank one.  The definition is independent of the
representatives.  By the classification of spherical objects and
\eqref{eq:intersection}, the map which forgets the local
system, flux, and parity gives a canonical graph isomorphism
\begin{equation}\label{eq:categorical-Schmutz}
 \Upsilon(\mathscr F_g)\cong\mathcal S(\Sigma_g).
\end{equation}
Indeed, rank-one branes on isotopic supports have the same
intersection-one neighbours, while two non-isotopic non-separating
curves are distinguished by such a neighbour
\cite[Lemma~7.6]{AurouxSmith}.

An enhanced autoequivalence preserves spherical objects, non-vanishing
of the Chern character, and dimensions of Floer cohomology.  It therefore
acts on \eqref{eq:categorical-Schmutz}.  The Euler pairing is the
algebraic intersection form, so this action preserves the orientation of
the surface.  In genus two one refines
\eqref{eq:categorical-relation} by retaining the parity of each
one-dimensional Floer group.  This gives the oriented Floer--Schmutz
graph.  The refinement separates the action of $\iota$ and lifts the quotient to
$\Gamma_2$.

Let $p\colon S(T\Sigma_g)\to\Sigma_g$ be the unit tangent bundle and
choose a one-form $\theta$ with $d\theta=p^*\omega$.  An oriented simple
closed curve is called \emph{balanced} when the integral of $\theta$ over its
tangent lift vanishes.  Every mapping class has a balanced symplectic
representative, unique up to Hamiltonian isotopy, and hence acts on the
Fukaya category.  This gives the geometric homomorphism
\[
 s_g\colon\Gamma_g\longrightarrow G_g.
\]

We summarize the geometric input in the form used below.

\begin{theorem}[{\cite{AurouxSmith,Schmutz}}]\label{thm:AS-input}
Let $X\in\mathscr F_g$ be spherical and suppose $\ch(X)\ne0$.  Then $X$
is represented by a homologically essential embedded simple closed curve
with a rank-one local system.  For rank-one branes $X_\gamma,X_\delta$ on
non-separating curves,
\[
 \rk_\La HF^*(X_\gamma,X_\delta)=1
 \quad\Longleftrightarrow\quad i(\gamma,\delta)=1.
\]
Consequently enhanced autoequivalences act on the Schmutz graph.  This
action gives a split epimorphism
\[
 \Psi_g\colon G_g\twoheadrightarrow\Gamma_g \quad\text{ for }g\geq3
\]
and an oriented split epimorphism
\[
 \Psi^{\mathrm{or}}_2\colon G_2\twoheadrightarrow\Gamma_2.
\]
The splittings are the geometric maps $s_g$ obtained from a balancing
datum.
\end{theorem}

The classification of spherical objects is by
\cite[Theorem~1.1]{AurouxSmith} and the Floer criterion is due to
\cite[Corollary~2.11]{AurouxSmith}.  The split quotients follow from
\cite[Proposition~7.7]{AurouxSmith} and the rigidity theorem for the
Schmutz graph \cite{Schmutz}.  For $g\geq3$ the graph is unoriented.
Therefore, $\Psi_g(\PiShift)=1$. In genus two, changing parity reverses every oriented vertex and the unique
mapping class with that action is the hyperelliptic involution, which implies $\Psi^{\mathrm{or}}_2(\PiShift)=\iota$.

We next record the subgroup which is visibly contained in the graph
kernel.  Every $\rho\in H_g$ has a factorization
\[
 \rho(\alpha)=q^{\ell(\alpha)}\rho_0(\alpha),
 \qquad
 \ell\in H^1(\Sigma_g;\mathbb R),\quad
 \operatorname{val}(\rho_0)=0.
\]
Transport by a symplectic isotopy of flux $\ell$ and tensor by the flat
local system of monodromy $\rho_0$.  Denote the resulting
autoequivalence by $R_\rho$.  Composition multiplies characters, and
application to branes on an integral homology basis shows that
$\rho\mapsto R_\rho$ is injective.  Moreover,
\begin{equation}\label{eq:conjugation}
 s_g(f)R_\rho s_g(f)^{-1}=R_{f\cdot\rho}.
\end{equation}
Thus the geometric section normalizes $H_g$, and $H_g$ lies in the graph kernel.  
At this stage we have proved normalization only by the geometric subgroup
$s_g(\Gamma_g)$.  Normality of $H_g$ in the whole group $G_g$ will follow
in Section~\ref{sec:proof}, after every element of $G_g$ has been put in
the normal form \eqref{eq:normal-form}.

The parity shift is central.  In the bimodule model it is the shifted
diagonal, and
\[
 \Delta[1]\otimes^\mathbb L M\simeq M[1]
 \simeq M\otimes^\mathbb L\Delta[1].
\]
It is disjoint from $H_g$, since
\begin{equation}\label{eq:shift}
 \ch(X[1])=-\ch(X),
\end{equation}
whereas every $R_\rho$ preserves Chern characters.

\begin{lemma}\label{lem:vertex-normalization}
Let $E$ be a rank-one brane on an oriented non-separating curve $\zeta$.
If an enhanced autoequivalence $F$ fixes the Schmutz vertex of $\zeta$,
then
\[
 F(E)\simeq (E,\lambda)[\epsilon]
 \qquad\text{for some }\lambda\in\La^\times,
 \quad \epsilon\in\Z/2.
\]
Here $\lambda$ combines the Novikov weight of an isotopy from the image
support to $\zeta$ with the holonomy of its rank-one local system, and
$\epsilon$ records whether the chosen orientation of $\zeta$ has been
reversed.
\end{lemma}

\begin{proof}
Theorem~\ref{thm:AS-input} represents $F(E)$ by a rank-one brane on an
embedded curve isotopic, as an unoriented curve, to $\zeta$.  Transport
through such an isotopy changes the brane by the exponential of its flux,
which is a Novikov factor, and by the holonomy of a flat rank-one local
system.  Their product is one element $\lambda\in\La^\times$.  With the
orientation fixed, these are all the remaining rank-one parameters.
Reversing the orientation of a one-dimensional brane changes its brane
parity.  In the two-periodic category this is the shift $[1]$.  This is
the geometric content of the rank-one family and balancing conventions in
\cite[Sections~2.2, 2.6 and 7.2]{AurouxSmith}.
\end{proof}

We now choose the finite configuration which will control the entire
kernel.  Let $\zeta_1,\ldots,\zeta_{2g}$ be oriented non-separating curves
with
\begin{equation*}
 i(\zeta_i,\zeta_j)=
 \begin{cases}
  1,&|i-j|=1,\\
  0,&|i-j|>1,
 \end{cases}
 \qquad
 [\zeta_i]\cdot[\zeta_{i+1}]=1.
\end{equation*}
Write $E_i$ for the brane with trivial rank-one local system on
$\zeta_i$, and let $\A_g$ be a minimal strictly unital model of the full
subcategory on these objects.

\begin{lemma}
The classes $[\zeta_1],\ldots,[\zeta_{2g}]$ form an integral basis of
$H_1(\Sigma_g;\Z)$, and $E_1,\ldots,E_{2g}$ split-generate
$\mathscr F_g$.  Consequently the inclusion
$\A_g\hookrightarrow\mathscr F_g$ is a Morita equivalence, and
\begin{equation}\label{eq:joint-evaluation}
 \bigoplus_{i=1}^{2g}\ev_{E_i}\colon
 \HH^1(\A_g)\longrightarrow
 \bigoplus_{i=1}^{2g}HF^1(E_i,E_i)
\end{equation}
is injective.
\end{lemma}

\begin{proof}
In the ordered family $[\zeta_1],\ldots,[\zeta_{2g}]$, the intersection
matrix is
\[
 J_{2g}=
 \begin{pmatrix}
 0&1&&&0\\
 -1&0&1&&\\
 &-1&0&\ddots&\\
 &&\ddots&\ddots&1\\
 0&&&-1&0
 \end{pmatrix}.
\]
Expanding successively along the first row and column gives
$\det J_{2g}=\det J_{2g-2}=\cdots=\det J_2=1$.  Since the intersection
form on $H_1(\Sigma_g;\Z)$ is unimodular of rank $2g$, the 
classes form an integral basis.  Split-generation is due to
\cite[Proposition~2.15]{AurouxSmith}.  Morita invariance and
\eqref{eq:HH-identification} identify $\HH^1(\A_g)$ with
$H^1(\Sigma_g;\La)$, while \eqref{eq:evaluation} becomes
\[
 \ev_{E_i}(a)=\langle a,[\zeta_i]\rangle u_i.
\]
Vanishing of all evaluations therefore means that $a$ pairs trivially
with an integral basis, and hence $a=0$.
\end{proof}

We next fix the enhanced model used below.  The notation for the action of
a derived Picard class and for the Taylor coefficients of a chosen
$A_\infty$-functor representative is the one fixed at the beginning of
this section.

\begin{proposition}\label{prop:models}
Let $F\in G_g$ and suppose that $F(E_i)\simeq E_i$ for every $i$.  After
these isomorphisms have been chosen, $F$ is represented on $\A_g$ by a
weakly invertible strictly unital $A_\infty$-functor
\[
 F=(F^0,F^1,F^2,\ldots)\colon\A_g\longrightarrow\A_g,
 \qquad F^0(E_i)=E_i.
\]
Two such functors represent the same element of $G_g$ if and only if they
are weakly equivalent.
\end{proposition}

\begin{proof}
Morita invariance transports an invertible bimodule across
$\A_g\hookrightarrow\mathscr F_g$.  Convolution with the transported
bimodule sends a representable module $h_{E_i}$ to the module represented
by the image of $E_i$.  Under the hypothesis, these image modules are
again represented by objects of the finite subcategory.  The internal
Hom description by $A_\infty$-functors and the object-preserving
strictification in \cite[Theorem~4.2 and Lemma~4.4]{OpperSurface} then
give the   functor.  The $A_\infty$ Yoneda embedding identifies
weak equivalences of functors with isomorphisms of their associated
right quasi-representable bimodules, which is precisely equality of the
corresponding classes in the derived Picard group.  See also
\cite{COS} for the internal Hom comparison.
\end{proof}

\section{Rigidity on the generating chain}\label{sec:rigidity}
In this section, we show that an autoequivalence fixing the objects
$E_i$ has only two possible enhanced classes. To achieve this, we first
analyze its action on the rank-one deformation of each object, and then
compare the resulting families through their relative Floer complex.

\begin{lemma}\label{lem:incidence}
Let $k$ be an algebraically closed field of characteristic zero, and
let $C$ be a perfect two-periodic complex over
$k[z^{\pm1},w^{\pm1}]$. Suppose that the points where the derived
fibre of $C$ is nonzero give a bijective correspondence between the two
copies of $k^\times$, and that every nonzero fibre has cohomology of
dimension one in each parity. The reduced support of $C$ is the graph
of $w=cz^\varepsilon$ for some $c\in k^\times$ and
$\varepsilon\in\{1,-1\}$.
\end{lemma}

\begin{proof}
The support is closed. Indeed, acyclicity of a finite projective
periodic complex is open, as can be seen from the ranks of its
differential matrices. Let $p$ be a closed point of the support.
After localizing at $p$ and cancelling contractible summands, the
complex has one generator in each parity. Its differentials are
multiplication by $f,h\in\mathfrak m_p$ with $fh=0$. The local ring is a
domain, so one of $f,h$ is zero. They cannot both vanish, since the
support would then contain an open subset of the two-dimensional
torus. Hence the reduced support is locally a divisor and has no
isolated points.

Every component dominates both coordinate factors, since a vertical
or horizontal component would violate uniqueness of a fibre. Two
components dominating the first factor would give two points above
a general $z$. Thus the support is an integral curve $Z$. Its map to
the first factor is generically one-to-one. In characteristic zero the
extension of function fields is separable, and hence $k(Z)=k(z)$.

Write $w=p(z)/q(z)$ in lowest terms. If $q(a)=0$ for some
$a\in k^\times$, then $p(a)\ne0$ and the identity
$q(z)w=p(z)$ would rule out a point of $Z$ over $a$. This contradicts
surjectivity onto the first factor. Thus $w\in k[z,z^{-1}]$. The same
argument applied to $w^{-1}$ shows that $w$ is a unit in this Laurent
polynomial ring, which implies that $w=cz^n$. Its bijectivity on closed
points gives $n=1$ or $n=-1$.
\end{proof}

\begin{lemma}\label{lem:rankone-hf}
Let $v,w\in\La^\times$.  For the orbit family of a
non-separating curve $\delta$, one has
\[
 HF^*(X_\delta(v),X_\delta(w))
 \cong
 \begin{cases}
 H^*(S^1;\La),&v=w,\\
 0,&v\ne w.
 \end{cases}
\]
\end{lemma}

\begin{proof}
Let $\Phi_z$ denote the rational $\Gm$-action defining the family, so
that $X_\delta(z)=\Phi_z(X)$, where $X=X_\delta(1)$.  Applying $\Phi_{v^{-1}}$ gives
\[
 HF^*(X_\delta(v),X_\delta(w))
 \cong HF^*(X,X_\delta(w/v)).
\]
It therefore suffices to determine those $z$ for which
$HF^*(X,X_\delta(z))$ is nonzero. By rebasing the orbit at a unitary
parameter, we may assume that $X$ carries the trivial local system.
This multiplies
both parameters by a common unit and does not change
their ratio.

Let $\tau_\delta$ be the symplectic Dehn twist associated
to $X$.  We first show that
\begin{equation}\label{eq:twist-fix}
 T_X(X_\delta(z))\simeq X_\delta(z)
 \qquad\text{for every }z\in\La^\times.
\end{equation}
Here $T_X$ is the spherical twist, identified with the
geometric Dehn twist by the exact triangle recalled in
\cite[Section~2.3]{AurouxSmith}.

Indeed, let $\delta_t$ be an oriented isotopy from
$\delta$ to the support of $X_\delta(z)$.
The isotopy $\tau_\delta(\delta_t)$ sweeps the same
signed area as $\delta_t$, since $\tau_\delta$ is
symplectic.  Moreover, $\tau_\delta$ preserves $\delta$
and its orientation.  Concatenating the reverse of
$\delta_t$ with $\tau_\delta(\delta_t)$ therefore gives
an isotopy from the support of $X_\delta(z)$ to its
image under $\tau_\delta$ which sweeps zero area.
A reparametrization along $\delta$ at the joining point
contributes no area.  The local system and spin structure
are preserved under this transport.  The Dehn twist
lift has no parity shift on its vanishing cycle, and
the same holds throughout its oriented isotopy class.
Thus \cite[Lemma~2.12]{AurouxSmith} proves
\eqref{eq:twist-fix}. 

Set $Y=X_\delta(z)$ and $V=HF^*(X,Y)$. Consider  the twist
triangle
\[
 V\otimes_\La X \longrightarrow Y
 \xrightarrow{\eta} T_X(Y)
 \longrightarrow (V\otimes_\La X)[1].
\]
Choose an isomorphism $T_X(Y)\simeq Y$ using
\eqref{eq:twist-fix}.  Since
$HF^0(Y,Y)=\La e_Y$, the resulting endomorphism
$\eta$ of $Y$ is either zero or invertible.
If it is invertible, then $V\otimes_\La X=0$ and
hence $V=0$.  If it is zero, the triangle splits to give
\[
 V\otimes_\La X\simeq Y\oplus Y[-1].
\]
The cohomological category is Hom-finite and
idempotent complete, so it is Krull--Schmidt.
Both $X$ and $Y$ are indecomposable, and
$V\otimes_\La X$ is a finite direct sum of copies
of $X$ and $X[1]$.  Consequently
$Y\simeq X$ or $Y\simeq X[1]$.
The latter possibility is excluded by
\[
 \ch(Y)=\ch(X)\ne0,
 \qquad
 \ch(X[1])=-\ch(X).
\]
Hence
\begin{equation}\label{eq:stabilizer}
 HF^*(X,X_\delta(z))\ne0
 \quad\Longleftrightarrow\quad
 X_\delta(z)\simeq X.
\end{equation}

Recall that $R=\La[z,z^{-1}]$ and consider
\[
 \mathcal H=
 \operatorname{RHom}_{\mathscr F_g\otimes_\La R}
 \bigl(X\otimes_\La R,\mathcal U_\delta\bigr).
\]
The orbit family is perfect by its construction in
\cite[Proposition~6.6 and Section~7.2]{AurouxSmith}.
Properness of $\mathscr F_g$ implies that $\mathcal H$
is a perfect two-periodic $R$-complex.  Its derived
fibre at $z$ computes $HF^*(X,X_\delta(z))$.
Hence
\[
 S=\{z\in\La^\times\mid HF^*(X,X_\delta(z))\ne0\}
\]
is Zariski closed in $\Gm$.  By
\eqref{eq:stabilizer}, it is also a
subgroup of $\La^\times$.

For $u\in U_\La$, the branes $X$ and $X_\delta(u)$
have the same support and differ only by the
rank-one local system of holonomy $u$.  Their Floer
cohomology is therefore
\[
 HF^*(X,X_\delta(u))
 \cong H^*(S^1;\La_u),
\]
where $\La_u$ denotes that local system on the circle.
Its cellular cochain complex has one generator in
each degree and differential multiplication by $u-1$.
It is acyclic unless $u=1$.  Thus
\(
 S\cap U_\La=\{1\}\). In particular $S$ is a proper closed subset of
$\Gm$, and is therefore finite.  Every element
$z\in S$ has finite order.  If $z^m=1$, then
$m\,\operatorname{val}(z)=0$, so $z\in U_\La$.
It follows that $S=\{1\}$.

Finally $HF^*(X,X)=H^*(S^1;\La)$, and applying the
initial reduction with $z=w/v$ proves the assertion.
\end{proof}

\begin{proposition}\label{prop:family}
Suppose that $F\in G_g$ sends the Schmutz vertex of an oriented curve
$\gamma$ to that of an oriented curve $\delta$. Equip the target brane family with the parity of $F(X_\gamma(1))$. Then $F$ induces an
algebraic isomorphism
\[
 \phi_{F,\gamma}\colon\mathfrak B_\gamma
 \xrightarrow{\sim}\mathfrak B_\delta.
\]
In algebraic coordinates it has the form
\begin{equation*}
 z\longmapsto cz^\varepsilon,
 \qquad c\in\La^\times,\quad \varepsilon\in\{1,-1\}.
\end{equation*}
If the target coordinate is based at
$F(X_\gamma(1))\simeq X_\delta(1)$, then $c=1$ and
\begin{equation}\label{eq:family-derivative}
 F_*(u_\gamma)=\varepsilon u_\delta.
\end{equation}
\end{proposition}

\begin{proof}
Represent $F$ by an invertible bimodule and set
$\mathcal V=F(\mathcal U_\gamma)$. The resulting family $\mathcal V$ is
perfect over $\La[z,z^{-1}]$, with fibre at $z$ given by
$F(X_\gamma(z))$. Consider the relative Floer complex
\[
C=\operatorname{RHom}_{\mathscr F_g\otimes S}
       (\mathcal V_z,\mathcal U_{\delta,w}).
\]
 over
$S=\La[z^{\pm1},w^{\pm1}]$. By \cite[Lemma~2.18]{AurouxSmith}, $C$ is a perfect two-periodic $S$-complex that commutes with derived
base change, whose fibre at $(z,w)$ is
$HF^*(F(X_\gamma(z)),X_\delta(w))$.

The Chern character of $F(X_\gamma(z))$ is independent of $z$ and is
nonzero. Its parity relative to the oriented curve $\delta$ is
therefore constant. By the geometricity theorem recalled in Section 2,
each fibre is a rank-one brane in the selected target family. By Lemma~\ref{lem:rankone-hf}, for every $z$ there is exactly one $w$ at which the fibre of
$C$ is nonzero, and that fibre has dimensions $(1,1)$. Applying the
same argument to $F^{-1}$ gives exactly one $z$ for each $w$.
Lemma~\ref{lem:incidence} shows that the reduced support is
$w=cz^\varepsilon$, with $\varepsilon=\pm1$. In particular the action
on rank-one brane parameters is regular and invertible. Normalizing the target family at $F(X_\gamma(1))$ gives $c=1$.

It remains to justify differentiation of this formula. Restrict to
the graph and set
$\mathcal W_z=\mathcal U_\delta(z^\varepsilon)$. In the local
one-generator-per-parity model used in the lemma, the differentials
vanish after restriction to the reduced graph. Near $z=1$ the
degree-zero part of $\operatorname{RHom}(\mathcal V,\mathcal W)$ is
therefore a line bundle and commutes with base change. The
isomorphism of the fibres at $1$ extends to a morphism of families
on a neighbourhood of $1$. Its cone has vanishing fibre at $1$. Testing
the cone against the finite split-generating chain and shrinking this
neighbourhood shows that the cone vanishes. Thus we obtain an isomorphism
of families near the base point.

The Kodaira--Spencer class of $\mathcal U_\gamma$ at $1$ is
$u_\gamma\otimes d\log z$, as recorded in Section 2. Then, functoriality
under convolution and the local isomorphism of families give
\[
 F_*(u_\gamma)\otimes d\log z
   =u_\delta\otimes d\log(z^\varepsilon)
   =\varepsilon u_\delta\otimes d\log z.
\]
Comparing coefficients gives \eqref{eq:family-derivative}.
\end{proof}

We now classify the autoequivalences which fix the generating
configuration.  The following normalization prevents a hidden character
factor from entering the comparison with the negative sign.

\begin{lemma}\label{lem:hyperelliptic-normalization}
Let $\iota_g$ be the hyperelliptic involution in the standard
double-cover picture of the generating chain. There is a character
$\eta\in H_g$ such that
\[
 \kappa_g=R_\eta\PiShift s_g(\iota_g)^{-1}
\]
fixes every $E_i$ and sends the rank-one parameter on $E_i$ to its
inverse. Moreover, $\kappa_g^2=1$.
\end{lemma}

\begin{proof}
The lifts of consecutive branch arcs form the chosen chain. The deck
involution preserves each lifted curve setwise, reverses its
orientation, and acts as $-I$ on homology. Set
$J=\PiShift s_g(\iota_g)^{-1}$. Orientation reversal followed by the
shift preserves the parity and support of each $E_i$. Write
$J(E_i)\simeq X_{\zeta_i}(\lambda_i)$, with
$\lambda_i\in\La^\times$. Since the $[\zeta_i]$ form an integral
basis, there is a unique $\eta\in H_g$ with
$\eta([\zeta_i])=\lambda_i^{-1}$ for every $i$. Thus
$\kappa_g=R_\eta J$ fixes each $E_i$.

The character conjugation formula and centrality of the shift give
$\kappa_gR_\rho\kappa_g^{-1}=R_{\rho^{-1}}$. As $\kappa_g$ fixes the
reference brane, it sends the character orbit with parameter $z$
to the orbit with parameter $z^{-1}$. Finally $J^2=1$ and
$JR_\eta J^{-1}=R_{\eta^{-1}}$, so
$\kappa_g^2=R_\eta R_{\eta^{-1}}=1$.
\end{proof}

\begin{proposition}\label{prop:object-rigidity}
Let $F\in G_g$ and suppose $F(E_i)\simeq E_i$ for
$1\leq i\leq2g$.  Then
\[
 F=1\quad\text{or}\quad F=\kappa_g
 \qquad\text{in }G_g.
\]
The induced action on every $HF^1(E_i,E_i)$ is respectively $+1$ or $-1$.
\end{proposition}

\begin{proof}
By Proposition~\ref{prop:models}, the chosen isomorphisms
$F(E_i)\simeq E_i$ give a weakly invertible strictly unital
$A_\infty$-functor $F\colon\A_g\to\A_g$ whose object map is the identity.
Changing those isomorphisms changes the representative by an invertible
$A_\infty$ natural transformation. Choose generators
\[
 HF^*(E_i,E_i)=\La e_i\oplus\La u_i,
 \qquad |e_i|=0,\quad |u_i|=1,
\]
and, for $1\leq i<2g$, generators $x_i$ and $y_i$ of the one-dimensional
spaces $HF^*(E_i,E_{i+1})$ and $HF^*(E_{i+1},E_i)$.  
The two non-zero local products are
\begin{equation}\label{eq:edge-products}
 \mu^2(y_i,x_i)=c_i u_i,
 \qquad \mu^2(x_i,y_i)=c'_i u_{i+1},
 \qquad c_i,c'_i\in\La^\times.
\end{equation}
On the cohomological category, we write
\[
 F^1(u_i)=d_i u_i,
 \qquad F^1(x_i)=a_i x_i,
 \qquad F^1(y_i)=b_i y_i,
\]
where every coefficient is non-zero and $F^1(e_i)=e_i$.  Applying the
$A_\infty$-functor equation to the products in
\eqref{eq:edge-products} gives
\[
 d_i=a_ib_i=d_{i+1}.
\]
Thus all $d_i$ are one scalar $d$.  Conjugating by scalar automorphisms
$c_i\id_{E_i}$ changes $a_i$ to $c_{i+1}a_ic_i^{-1}$.  Since the
$A_{2g}$-graph is a tree, choose $c_1=1$ and then choose $c_{i+1}$
recursively so that every $a_i$ becomes one. The product identities then
give
\begin{equation}\label{eq:linear-form}
 F^1(e_i)=e_i,\qquad F^1(u_i)=d u_i,\qquad
 F^1(x_i)=x_i,\qquad F^1(y_i)=d y_i.
\end{equation}

Apply Proposition~\ref{prop:family} to the brane torsor through $E_i$.
The base point is fixed, so the induced map is $z\mapsto z^{\varepsilon_i}$
with $\varepsilon_i=\pm1$, and its logarithmic derivative is the action of
$F^1$ on $u_i$.  Hence $d=\varepsilon_i\in\{1,-1\}$ for every $i$.

It remains to prove that the first Taylor coefficient determines the
enhanced functor.  Let $C(\A_g)$ be the shifted Hochschild cochain complex
and let
\[
 W_rC(\A_g)=\{c=(c^0,c^1,\ldots)\mid c^j=0\text{ for }j<r\}.
\]
We use the image filtration
\begin{equation*}
 W_2\HH^1(\A_g):=
 \im\bigl(H^{\bar0}(W_2C(\A_g))\to H^{\bar0}(C(\A_g))\bigr).
\end{equation*}
If $[c]\in W_2\HH^1(\A_g)$, choose a cocycle representative with
$c^0=c^1=0$.  Evaluation at $E_i$ is represented by $c^0_{E_i}$, and
therefore every evaluation of $[c]$ vanishes.  The injectivity of
\eqref{eq:joint-evaluation} gives
\begin{equation}\label{eq:W2-vanishes}
 W_2\HH^1(\A_g)=0.
\end{equation}
The hypotheses of Proposition~\ref{prop:rigid-periodic} hold
for $\A_g$ by \eqref{eq:joint-evaluation} and
$HF^{\bar0}(E_i,E_i)=\La e_i$. It follows that an object-fixing
autofunctor with $F^1=\id$ is weakly equivalent to the identity.  Consequently two object-fixing functors with
the same first Taylor coefficient after scalar gauge define the same
element of $G_g$.

For $d=1$, the normal form \eqref{eq:linear-form} is the first
Taylor coefficient of the identity, and hence $F=1$.  For $d=-1$, the
geometric description preceding the proposition and
Proposition~\ref{prop:family} show that $\kappa_g^1(u_i)=-u_i$.
Using the same tree gauge, we may assume
that its entire first coefficient is given by the normal form
\eqref{eq:linear-form} with $d=-1$. Consequently,
$F\kappa_g^{-1}$ has identity first Taylor coefficient, hence is
trivial. Thus we conclude that
$F=\kappa_g$.

\end{proof}

\section{Proof of the main theorem}\label{sec:proof}

In this section, we first compute the graph kernels, and then recover the full group, which completes the main result. 

\begin{proof}[Proof of Theorem~\ref{thm:main}]
First consider $g\geq3$ and assume $F\in\ker\Psi_g$.  Since $F$ fixes the Schmutz
vertex of every $E_i$, Lemma~\ref{lem:vertex-normalization} gives
\begin{equation}\label{eq:kernel-images}
 F(E_i)\simeq(E_i,\lambda_i)[\epsilon_i],
 \qquad \lambda_i\in\La^\times,\quad
 \epsilon_i\in\Z/2.
\end{equation}
For adjacent objects the Floer group is one-dimensional in a definite
parity.  A degree-zero
autoequivalence preserves that parity, whereas shifting the two objects
changes it by $\epsilon_{i+1}-\epsilon_i$.  Hence
$\epsilon_i=\epsilon_{i+1}$, and connectedness of the chain gives one
common parity $\epsilon$.

Because $[\zeta_1],\ldots,[\zeta_{2g}]$ is an integral homology basis,
there is a unique character $\rho\in H_g$ satisfying
\[
 \rho([\zeta_i])=\lambda_i\qquad\text{ for }1\leq i\leq2g.
\]
The normalized element \(F_0=R_\rho^{-1}\PiShift^{-\epsilon}F\)
fixes all $E_i$. Proposition~\ref{prop:object-rigidity} then gives
$F_0\in\{1,\kappa_g\}$.  But
\[
 \Psi_g(\kappa_g)=\iota_g^{-1}\ne1,
\]
because $\Psi_g(\PiShift)=1$ and the hyperelliptic mapping class is
non-trivial. Since $F_0$ still belongs to the kernel, the second
possibility is excluded.  Thus $F=R_\rho\PiShift^\epsilon$ and by \eqref{eq:shift}, we have
\begin{equation}\label{eq:kernel-high}
 \ker\Psi_g=H_g\times\langle\PiShift\rangle.
\end{equation}

Now let $g=2$ and assume $F\in\ker\Psi^{\mathrm{or}}_2$.  In the oriented graph
the common reversal in \eqref{eq:kernel-images} is recorded, so the kernel
condition forces the common parity to be zero.  Removing the unique
character again gives an element $F_0$ which fixes the four chain objects.
Proposition~\ref{prop:object-rigidity} yields
$F_0\in\{1,\kappa_2\}$.  Both alternatives occur, since
\[
 \Psi^{\mathrm{or}}_2(\kappa_2)
 =\Psi^{\mathrm{or}}_2(\PiShift)\iota^{-1}=1.
\]
It follows that
\begin{equation}\label{eq:kernel-two}
 \ker\Psi^{\mathrm{or}}_2=H_2\sqcup H_2\kappa_2.
\end{equation}
The shift is central, while $s_2(\iota)$ acts on characters by inversion
because $\iota_*=-I$ on $H_1(\Sigma_2;\Z)$.  Therefore
\[
 \kappa_2R_\rho\kappa_2^{-1}=R_{\rho^{-1}}.
\]
Moreover, Lemma~\ref{lem:hyperelliptic-normalization} gives $\kappa_2^2=1$.  Finally we claim that $\kappa_2\notin H_2$. Otherwise, its conjugation action on the abelian group $H_2$ would be
trivial, contrary to inversion.  Equation \eqref{eq:kernel-two} is
therefore the generalized dihedral group
$H_2\rtimes_{\rho\mapsto\rho^{-1}}\Z/2$.

It remains to identify $G_g$ and to prove normality of $H_g$ in $G_g$.
For $g\geq3$, take an arbitrary
$F\in G_g$, put $f=\Psi_g(F)$, and apply \eqref{eq:kernel-high} to
$Fs_g(f)^{-1}$.  There are unique $\eta\in H_g$ and
$e\in\Z/2$ such that
\begin{equation}\label{eq:normal-form}
 F=R_\eta\PiShift^e s_g(f).
\end{equation}
For every $\rho\in H_g$, due to the abelianness of the character subgroup,
centrality of $\PiShift$, and \eqref{eq:conjugation}, we have
\begin{align*}
 F R_\rho F^{-1}
 &=R_\eta\PiShift^e s_g(f)R_\rho
   s_g(f)^{-1}\PiShift^{-e}R_\eta^{-1}\\
 &=R_\eta R_{f\cdot\rho}R_\eta^{-1}
  =R_{f\cdot\rho},
\end{align*}
which implies $H_g\triangleleft G_g$.  The multiplication law in the
normal form \eqref{eq:normal-form} is that of
\[
 G_g\cong(H_g\rtimes\Gamma_g)\times\langle\PiShift\rangle.
\]
For $g=2$, the oriented splitting and \eqref{eq:kernel-two} give a unique
expression \[
 R_\rho\kappa_2^e s_2(f)
=R_{\rho\eta^e}s_2(\iota^{-e}f)\PiShift^e.
\]  
Thus $G_2$ is generated by $H_2$, the mapping-class section, and the
central parity shift, with the conjugation relation
\eqref{eq:conjugation}. If $R_\rho s_2(f)\PiShift^e=1$, applying the oriented
quotient gives $f\iota^e=1$.  When $e=1$, the remaining equality would put
$\kappa_2$ in $H_2$. Consequently
\[
 G_2\cong(H_2\rtimes\Gamma_2)\times\langle\PiShift\rangle.
\]
In particular every element has a unique normal form
$R_\eta s_2(f)\PiShift^e$.  Repeating the   conjugation
calculation with this normal form proves
$FR_\rho F^{-1}=R_{f\cdot\rho}$ also in genus two.  Hence
$H_2\triangleleft G_2$ as well.
\end{proof}

The unoriented genus-two Schmutz map has target
$\Gamma_2/\langle\iota\rangle$.  In the coordinates of
Theorem~\ref{thm:main}, the two genus-two maps are
\[
 \Psi^{\mathrm{or}}_2(\rho,f,\PiShift^e)=f\iota^e,
 \qquad
 \Psi^{\mathrm{un}}_2(\rho,f,\PiShift^e)
 =f\bmod\langle\iota\rangle.
\]
Hence the unoriented kernel is
$\Dih(H_2)\times\langle\PiShift\rangle$.

Let $\mathscr F_1$ be the two-periodization of
the graded Fukaya category of the torus. 
Homological mirror symmetry identifies the second with the
two-periodization of $\Perf(E_Q)$ for the Tate elliptic curve $E_Q$
\cite{PolishchukZaslow,LekiliPerutz}. According to \cite{HKR,Yekutieli}, we have
\[
 \HH^1(\Perf(E_Q))
 \cong
 \bigoplus_{p+q=1}
 H^q(E_Q,\wedge^p T_{E_Q})
 =
 H^0(E_Q,T_{E_Q})
 \oplus
 H^1(E_Q,\mathcal O_{E_Q}),
\]
where both summands are one-dimensional. For
$g\geq2$, the closed--open isomorphism gives
$\HH^{\mathrm{odd}}(\mathscr F_g)\cong H^1(\Sigma_g;\La)$.  Consequently
\[
 \dim_{\La}\HH^{\mathrm{odd}}(\mathscr F_g)=2g
 \qquad\text{ for }g\geq1.
\]
Since Hochschild cohomology is invariant under enhanced Morita equivalence,
an equivalence $\mathscr F_g\simeq\mathscr F_h$ implies $g=h$.

\begin{corollary}
\label{cor:reconstruct}
Let $g,h\geq1$.  If $G_g\cong G_h$ as abstract groups, then $g=h$.
\end{corollary}
\begin{proof}
Set
\[
 D_1:=E_Q(\La)\times\widehat E_Q(\La),
 \qquad D_g:=H_g\quad\text{ for }g\geq2,
\]
where $\widehat E_Q=\Pic^0(E_Q)$. Following \cite{OrlovAbelian}, there exists
the following exact sequence
\[
1\longrightarrow\operatorname{Aut}(E_Q,0)
 \longrightarrow G_1/D_1
 \longrightarrow\operatorname{SL}(2,\Z)
 \longrightarrow1.
\]
Combining with Theorem~\ref{thm:main}, $D_g$ is the largest divisible abelian normal subgroup of $G_g$.  The
groups $D_g$ are divisible because $\La^\times$ is divisible and
$E_Q(\La)=\La^\times/Q^{\Z}$, while $E_Q\simeq\widehat E_Q$.  The relevant quotients are a finite
extension of $\operatorname{SL}(2,\Z)$ when $g=1$ and
$\Gamma_g\times\Z/2$ when $g\geq2$.  Abelian subgroups of
$\operatorname{SL}(2,\Z)$ are virtually cyclic, and  abelian subgroups of
$\Gamma_g$ are finitely generated \cite{BLM}.  Hence the image of a
divisible abelian subgroup in either quotient is trivial.  Thus $D_g$ is
characteristic.

An isomorphism $G_g\cong G_h$ identifies $D_g$ with $D_h$.  For every
odd prime $p$ one has
\[
 \dim_{\mathbb F_p}D_g[p]
 =
 \begin{cases}
 4,&g=1,\\
 2g,&g\geq2.
 \end{cases}
\]
It follows that either $g=h$ or $\{g,h\}=\{1,2\}$.  The latter possibility
is excluded by
\[
 \vcd(G_1/D_1)=1,
 \qquad
 \vcd(G_2/D_2)=\vcd(\Gamma_2)=3
\]
due to \cite{Harer}.  Hence $g=h$.
\end{proof}

\appendix

\section{The two-periodic Hochschild exponential}\label{app:exponential}

In this section, we record the two-periodic form of Hochschild
integration that is used in Proposition~\ref{prop:object-rigidity}. For the general theory of Hochschild integration for
$\Z$-graded categories, we refer to
\cite{OpperIntegration,OpperSurface}.

Let $\C$ be a small, uncurved, strictly unital two-periodic
$A_\infty$-category over a characteristic-zero field.  Its shifted
Hochschild cochain space is given by
\[
 C^{\bar\epsilon}(\C)
 =\prod_{d\geq0}\prod_{X_0,\ldots,X_d}
 C^{\bar\epsilon}_d(X_0,\ldots,X_d),
 \]
 where
 \[
 C^{\bar\epsilon}_d(X_0,\ldots,X_d)
 =\Hom^{\bar\epsilon}\!\left(
 \C(X_{d-1},X_d)[1]\otimes\cdots\otimes\C(X_0,X_1)[1],
 \C(X_0,X_d)[1]\right).
\]
For $c=(c^0,c^1,\ldots)$, let
$W_rC(\C)\coloneq\{c\mid c^d=0\text{ for }d<r\}$.  The Hochschild differential is defined as
$\delta_\C=[\mu_\C,-]$, and denote by 
\[
 W_2\HH^1(\C)\coloneq
 \im\bigl(H^{\bar0}(W_2C(\C),\delta_\C)
 \longrightarrow H^{\bar0}(C(\C),\delta_\C)\bigr).
\]

In the suspended bar convention, a normalized cochain vanishes when
an input is a strict unit. The $A_\infty$ structure is an odd bar
coderivation $b$, and its Hochschild differential is $[b,-]$. An
endomorphism of the bar construction whose unary term is the identity
is pronilpotent in tensor length. Its logarithm and exponential are
finite on each fixed tensor length. We use these operations only for
the following rigidity statement.

\begin{proposition}\label{prop:rigid-periodic}
Let $\C$ be a minimal, uncurved, strictly unital two-periodic
$A_\infty$-category over a field $k$ of characteristic zero. Suppose
that $H^{\bar0}\operatorname{hom}(X,X)=k e_X$ for every object and that
the joint evaluation map
\[
 \HH^{\mathrm{odd}}(\C)\longrightarrow
 \prod_X H^{\bar1}\operatorname{hom}(X,X)
\]
is injective. Every strictly unital $A_\infty$-autofunctor which fixes
all objects and has first Taylor coefficient equal to the identity
is weakly equivalent to the identity.
\end{proposition}

\begin{proof}
Write $B(F)$ for the bar map of the functor. The operator
$N=B(F)-1$ strictly lowers tensor length. Hence
\[
 c=\log B(F)=\sum_{m\geq1}\frac{(-1)^{m+1}}{m}N^m
\]
is defined in each tensor length. To verify that $c$ is a
coderivation, form the polynomial family
$B(F)^t=\sum_m\binom{t}{m}N^m$. In any fixed tensor length the
coalgebra identity holds when $t$ is a nonnegative integer, hence
as a polynomial identity. Differentiate at $t=0$. Since $B(F)$
commutes with $b$, so does $c$. Its Taylor components in arities zero
and one vanish, so $c$ is a Hochschild cocycle in $W_2$.

The class of $c$ has zero evaluation at every object. By the
assumption on joint evaluation there is an odd suspended cochain $h$
such that $c=[b,h]$. Its nullary component is an even scalar
endomorphism, say $h_X^0=\lambda_Xe_X$. Because $c^1=0$ and the
category is minimal, $(\lambda_Y-\lambda_X)a=0$ for every
$a\in\operatorname{hom}(X,Y)$. The cochain consisting of these scalar
nullary components is closed. For its higher differential terms this
uses strict unitality and uncurvedness. Subtracting it from $h$ gives
\[
 c=[b,h],\qquad h^0=0.
\]

We use the integration map for the completed Hochschild
brace algebra \cite[Theorems~6.12 and~6.23]{OpperIntegration}.
It associates to a class in $W_2\HH^1(\C)$ the weak
equivalence class of its exponential $A_\infty$-functor.
For the cocycle $c=\log B(F)$ constructed above, this
exponential functor is $F$.

The same construction applies to the present two-periodic
category.  The normalized Hochschild cochains are complete
for the arity filtration.  In each fixed arity the brace
operations, exponential and homotopies are finite sums.
Their Koszul signs depend only on the parities of the
inputs, so the identities used in
\cite[Theorems~6.12 and~6.23]{OpperIntegration}
hold with degrees reduced modulo two.  In particular,
the integration map is defined on $W_2\HH^1(\C)$ and
depends only on the Hochschild class.

We have shown that $c=[b,h]$, with $h^0=0$.  Hence
$[c]=0$ in $W_2\HH^1(\C)$, and therefore
\[
 [F]=\operatorname{Exp}_{\C}([c])
    =\operatorname{Exp}_{\C}(0)
    =[\id_{\C}]
\]
in the group of weak equivalence classes of
$A_\infty$-autofunctors.  This proves the proposition.
\end{proof}


\begin{thebibliography}{99}

\bibitem{Abouzaid}
M.~Abouzaid,
\emph{On the Fukaya categories of higher genus surfaces},
Adv. Math. \textbf{217} (2008), no.~3, 1192--1235.

\bibitem{AurouxSmith}
D.~Auroux and I.~Smith,
\emph{Fukaya categories of surfaces, spherical objects, and mapping class
groups}, Forum Math. Sigma \textbf{9} (2021), Paper No.~e26, 50 pp.

\bibitem{BLM}
J.~Birman, A.~Lubotzky, and J.~McCarthy,
\emph{Abelian and solvable subgroups of the mapping class groups},
Duke Math. J. \textbf{50} (1983), 1107--1120.

\bibitem{COS}
A.~Canonaco, M.~Ornaghi, and P.~Stellari,
\emph{Localizations of the category of $A_\infty$ categories and internal
Homs}, Doc. Math. \textbf{24} (2019), 2463--2492.


\bibitem{Harer}
J.~Harer,
\emph{The virtual cohomological dimension of the mapping class group of an
orientable surface}, Invent. Math. \textbf{84} (1986), 157--176.

\bibitem{HKK}
F.~Haiden, L.~Katzarkov, and M.~Kontsevich,
\emph{Flat surfaces and stability structures},
Publ. Math. Inst. Hautes \`Etudes Sci. \textbf{126} (2017), 247--318.

\bibitem{KellerDG}
B.~Keller,
\emph{On differential graded categories},
International Congress of Mathematicians, Vol.~II,
Eur. Math. Soc., Z\"urich, 2006, 151--190.

\bibitem{LekiliPerutz}
Y.~Lekili and T.~Perutz,
\emph{Arithmetic mirror symmetry for the $2$-torus},
arXiv:1211.4632.

\bibitem{LekiliPolishchuk}
Y.~Lekili and A.~Polishchuk,
\emph{Derived equivalences of gentle algebras via Fukaya categories},
Math. Ann. \textbf{376} (2020), no.~1--2, 187--225.

\bibitem{OrlovAbelian}
D.~Orlov,
\emph{Derived categories of coherent sheaves on abelian varieties and
equivalences between them},
Izv. Math. \textbf{66} (2002), no.~3, 569--594.

\bibitem{PolishchukZaslow}
A.~Polishchuk and E.~Zaslow,
\emph{Categorical mirror symmetry: the elliptic curve},
Adv. Theor. Math. Phys. \textbf{2} (1998), no.~2, 443--470.

\bibitem{OpperIntegration}
S.~Opper,
\emph{Integration of Hochschild cohomology, derived Picard groups and
uniqueness of lifts}, arXiv:2405.14448.

\bibitem{OpperSurface}
S.~Opper,
\emph{Autoequivalences of Fukaya categories of surfaces and graded gentle
algebras}, arXiv:2510.11543.

\bibitem{PascaleffSibilla}
J.~Pascaleff and N.~Sibilla,
\emph{Fukaya categories of higher-genus surfaces and pants decompositions},
arXiv:2103.03366.

\bibitem{Schmutz}
P.~Schmutz Schaller,
\emph{Mapping class groups of hyperbolic surfaces and automorphism groups
of graphs}, Compositio Math. \textbf{122} (2000), 243--260.

\bibitem{SeidelBook}
P.~Seidel,
\emph{Fukaya categories and Picard--Lefschetz theory},
Zurich Lectures in Advanced Mathematics,
European Mathematical Society, Z\"urich, 2008.

\bibitem{SeidelFlux}
P.~Seidel,
\emph{Abstract analogues of flux as symplectic invariants},
M\'em. Soc. Math. Fr. (N.S.) \textbf{137} (2014), 135 pp.

\bibitem{Toen}
B.~To\"en,
\emph{The homotopy theory of dg-categories and derived Morita theory},
Invent. Math. \textbf{167} (2007), no.~3, 615--667.

\bibitem{HKR}
G.~Hochschild, B.~Kostant, and A.~Rosenberg,
\emph{Differential forms on regular affine algebras},
Trans. Amer. Math. Soc. \textbf{102} (1962), 383--408.

\bibitem{Yekutieli}
A.~Yekutieli,
\emph{The continuous Hochschild cochain complex of a scheme},
Canad. J. Math. \textbf{54} (2002), no.~6, 1319--1337.

\end{thebibliography}
\end{document}